\documentclass[11pt,english]{article}

\usepackage{amsmath,amsthm,latexsym,graphicx,epstopdf,cite,amssymb,authblk,float,subcaption,cleveref}
\usepackage[margin=3cm]{geometry}

\newcommand\blfootnote[1]{
	\begingroup\renewcommand\thefootnote{}\footnote{#1}
	\addtocounter{footnote}{-1}
	\endgroup}

\newtheorem{theorem}{Theorem}[section]

\newtheorem{corollary}[theorem]{Corollary}
\newtheorem{remark}{Remark}[section]

\numberwithin{equation}{section}

\title{\bf Extremal conjugated unicyclic and bicyclic graphs with respect to total-eccentricity index}

\author{Mehar Ali Malik}
\author{Rashid Farooq\thanks{Corresponding author.}~}
\affil{School of Natural Sciences,\\
	National University of Sciences and Technology, Sector H-12, Islamabad, Pakistan}

\date{}

\begin{document}
\maketitle
\blfootnote{\raggedright Email addresses:
alies.camp@gmail.com, mehar.ali@sns.nust.edu.pk (M. A. Malik),
farook.ra@gmail.com (R. Farooq).}

\begin{abstract}
Let $G$ be a molecular graph.
The total-eccentricity index of graph $G$ is defined as the sum of eccentricities of all vertices of $G$.
The extremal trees, unicyclic and bicyclic graphs, and extremal conjugated trees with respect to total-eccentricity index are known.
In this paper, we extend these results and study the extremal conjugated unicyclic and bicyclic graphs with respect to total-eccentricity index.
\end{abstract}
\begin{quote}
\end{quote}
	
\begin{quote}
\end{quote}	
	
\section{Introduction}
%

Let $G$ be an $n$-vertex molecular graph with vertex set $V(G)$ and edge set $E(G)$.
The vertices and edges of $G$ respectively correspond to atoms and chemical bonds between atoms.
A topological index $T$ is a numerical quantity associated with the chemical structure of a molecule.
The aim of such association is to correlate these indices with various physico-chemical properties of a chemical compound.
An edge between two vertices $u,v\in V(G)$ is denoted by $uv$.
The order and size of $G$ are respectively the cardinalities $|V(G)|$ and $|E(G)|$.
The neighbourhood $N_G(v)$ of a vertex $v$ in $G$ is the set of vertices adjacent to $v$.
A simple graph is a graph without loops and multiple edges.
All graphs considered in this paper are simple.
The degree $d_G(v)$ of a vertex $v$ in $G$ is the cardinality $|N_G(v)|$.
A graph $G$ is called a $k$-regular graph if $d_G(v)=k$ for all $v\in V(G)$.
A vertex of degree $1$ is called a pendant vertex.
Let $P_n$, $C_n$ and $K_n$ respectively denote an $n$-vertex path, cycle and a complete graph.
An $n$-vertex complete bipartite graph and $n$-vertex star is respectively denoted by $K_{a,b}$ and $K_{1,n-1}$ (or simply $S_n$), where $a+b=n$.
A $(v_1,v_n)$-path with vertex set $\{v_1,v_2,\ldots,v_n\}$ is denoted by $v_1v_2\ldots v_n$.
A graph $G$ is said to be connected if there exists a path between every pair of vertices in $G$.
A maximal connected subgraph of a graph is called a component.
A vertex $v\in V(G)$ is called a cut-vertex if deletion of $v$, along with the edges incident on it, increases the number of components of $G$.
A maximal connected subgraph of a graph without any cut-vertex is called a block.
%
A tree is a connected graph containing no cycle.
Thus, an $n$-vertex tree is a connected graph with exactly $n-1$ edges.
An $n$-vertex unicyclic graph is a simple connected graph which contains $n$ edges.
Similarly, an $n$-vertex bicyclic graph is a simple connected graph which contains $n+1$ edges.
	

A matching $M$ in a graph $G$ is a subset of edges of $G$ such that no two edges in $M$ share a common vertex.
A vertex $u$ in $G$ is said to be $M$-saturated if an edge of $M$ is incident with $u$.
A matching $M$ is said to be perfect if every vertex in $G$ is $M$-saturated.
A conjugated graph is a graph that contains a perfect matching.
In graphs representing organic compounds, perfect matchings
correspond to Kekul\'{e} structures, playing an important role in analysis of the resonance energy and stability of hydrocarbons~\cite{Gutman}.
The distance $d_G(u,v)$ between two vertices $u,v\in V(G)$ is defined as the length of a shortest path between $u$ and $v$ in $G$.
If there is no path between vertices $u$ and $v$ then $d_G(u,v)$ is defined to be $\infty$.
The eccentricity $e_G(v)$ of a vertex $v\in V(G)$ is defined as the largest distance from $v$ to any other vertex in $G$.
The diameter ${\rm diam}(G)$ and radius ${\rm rad}(G)$ of a graph $G$ are respectively defined as:
\begin{eqnarray}
\label{rad} {\rm rad}(G) &=& \min_{v\in V(G)}e_G(v),\\
\label{diam} {\rm diam}(G) &=& \max_{v\in V(G)}e_G(v).
\end{eqnarray}
A vertex $v\in V(G)$ is said to be central if $e_G(v) = {\rm rad}(G)$.
The graph induced by all central vertices of $G$ is called the center of $G$, denoted as $C(G)$.
A vertex $w$ is called an eccentric vertex of a vertex $v$ in $G$ if $d_G(v,w) = e_G(v)$. The set of all eccentric vertices of $v$ in a graph $G$ is denoted by $E_G(v)$.
	
The first topological index was introduced by Wiener~\cite{W1947} in 1947, to calculate the boiling points of paraffins. In 1971, Hosoya~\cite{H71} defined the notion of Wiener index for any graph as the half sum of distances between all pairs of vertices.
The average-eccentricity of an $n$-vertex graph $G$ was defined in 1988 by Skorobogatov and Dobrynin~\cite{Skoro1988} as:
\begin{equation}\label{average}
avec(G) = \frac{1}{n}\sum_{u\in V(G)} e_G(u).
\end{equation}
In the recent literature, a minor modification of average-eccentricity index $avec(G)$ is used and referred as total-eccentricity index $\tau(G)$. It is defined as:
\begin{equation}\label{sigma}
\tau(G) = \sum_{u\in V(G)} e_G(u).
\end{equation}
The eccentric-connectivity index and the Randi\'c index of a graph $G$ are defined respectively by 
$\xi(G) = \sum_{v\in V(G)} d_G(v) e_G(v)$ and
$R(G)= \sum_{uv\in E(G)}(d_G(u)d_G(v))^{-\frac{1}{2}}$.
Liang and Liu~\cite{LiLiu2016} proved a conjecture on the relation between the average-eccentricity and Randi\'c index.
Dankelmann and Mukwembi~\cite{DankleMuk2004}
obtained upper bounds on the average-eccentricity in terms of several graph parameters.
Smith et al.~\cite{Smith2016} studied the
extremal values of total-eccentricity index in trees.
Ilic~\cite{Ilic12} studied some extremal graphs with respect to average-eccentricity.
Farooq et al.~\cite{Farooq2017} studied the extremal unicyclic and bicyclic graphs and extremal conjugated trees with respect to total-eccentricity index.
For more details on topological indices of graphs and networks, the author is referred to~\cite{akhter-farooq2019,Doslic}.
In this paper, we extend the results of \cite{Farooq2017} to conjugated unicyclic and bicyclic graphs.
	%
	%
	%
	%
	%
	%
	%
	%
	
	For some special families of graphs of order $n\geq 4$, the total-eccentricity index is given as follows:
	\begin{enumerate}
		\item For a $k$-regular graph $G$, we have $\tau(G) = \frac{\xi(G)}{k}$,
		\item $\tau(K_n) = n$,
		\item $\tau(K_{m,n}) = 2(m+n)$, $m,n\geq 2$,
		\item The total-eccentricity index of a star $S_n$, a cycle $C_n$ and a path $P_n$ is given by
		\begin{eqnarray}
		\label{tau Sn}
		\tau(S_n) &=& 2n - 1,\\
		\tau(C_n) &=& \begin{cases}
		\frac{n}{2}& \mbox{if~ } n\equiv 0 (\bmod 2)\\
		\frac{n-1}{2} & \mbox{if~ } n\equiv 1 (\bmod 2),
		\end{cases}\\
		\label{tau Pn}
		\tau(P_n) &=& \begin{cases}
		\frac{3n^2}{4} - \frac{n}{2} & \mbox{if~ } n\equiv 0 (\bmod 2)\\
		\frac{3n^2}{4} - \frac{n}{2} - \frac{1}{4} & \mbox{if~ } n\equiv 1 (\bmod 2).
		\end{cases}
		\end{eqnarray}
	\end{enumerate}

Let $\{v_1,v_2,\ldots,v_n\}$ be the vertices of a path $P_n$. Let $U_2$ be a unicyclic graph obtained from $P_n$ by joining $v_1$ and $v_3$ by an edge. Similarly, let $B_2$ be a bicyclic graph obtained from $P_n$ by joining $v_1$ with two vertices $v_3$ and $v_4$. Note that when $n\equiv 0\pmod 2$, the graphs $U_2$ and $B_2$ are conjugated and are denoted by 
$\overline{U}_2$ and $\overline{B}_2$, respectively. In Figure~\ref{max conj}, we give two $7$-vertex bicyclic graphs $U_2$ and $B_2$.
For $n\equiv 0\pmod 2$, let $S_n^*$ be an $n$-vertex conjugated tree obtained by identifying one vertex each from $\frac{n-2}{2}$ copies of $P_3$ and deleting a single pendent vertex. Let $v$ be the unique central vertex of $S_n^*$.
Let $\overline{U}_1$ be a conjugated unicyclic graph obtained from $S_n^*$ by adding an edge between $v$ and any vertex not adjacent to $v$. In a similar fashion, let $\overline{B}_1$ be a conjugated bicyclic graph obtained from $S_n^*$ by adding two edges between $v$ and any two vertices not adjacent to $v$ (see Figure~\ref{ext conj}).
\begin{figure}[h]
	\centering
	\includegraphics[height=2cm,width=5cm]{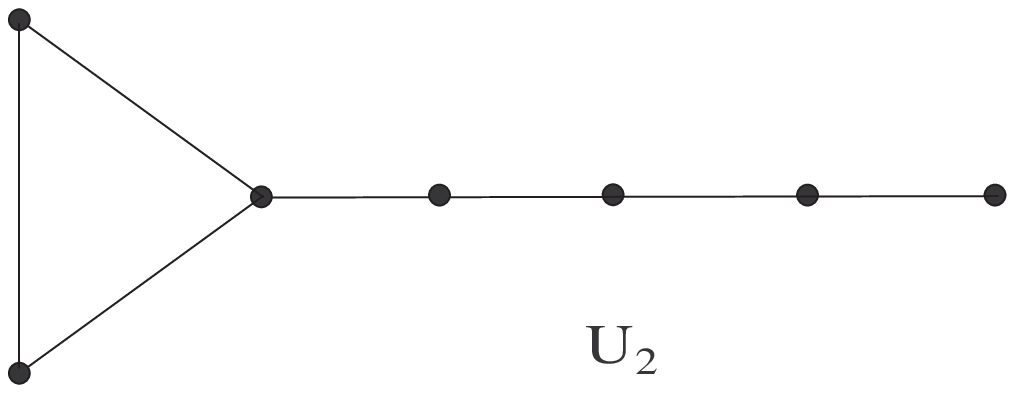}
	\includegraphics[height=2cm,width=5.3cm]{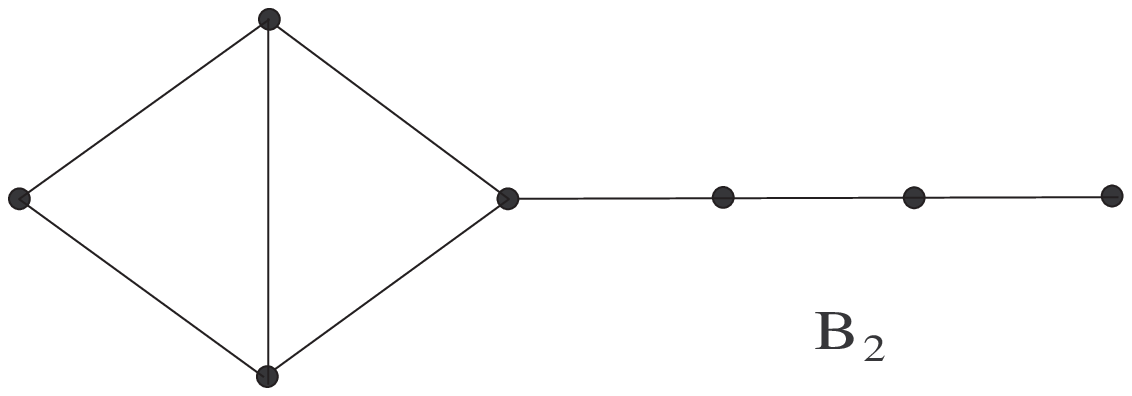}
	\caption{The $7$-vertex unicyclic and bicyclic graphs $U_2$ and $B_2$.}\label{max conj}
\end{figure}
\begin{figure}[h!]
\centering
\includegraphics[height=2.7cm,width=4.5cm]{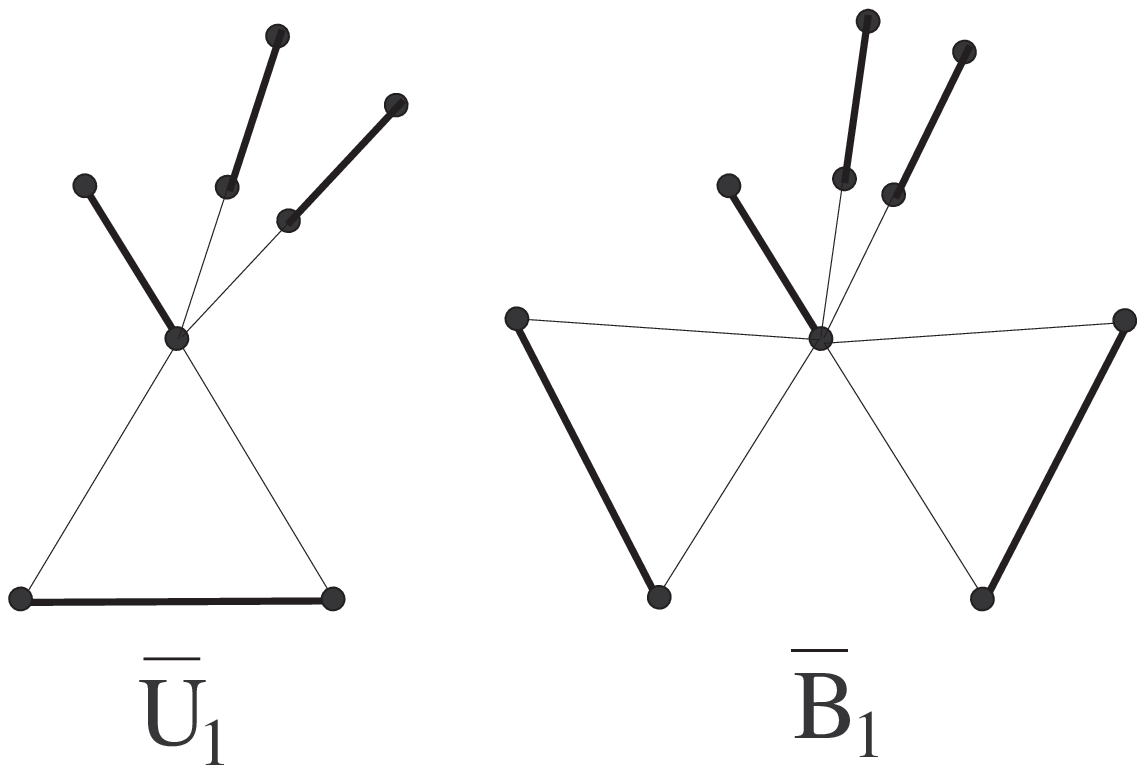}
\includegraphics[height=2.7cm,width=7cm]{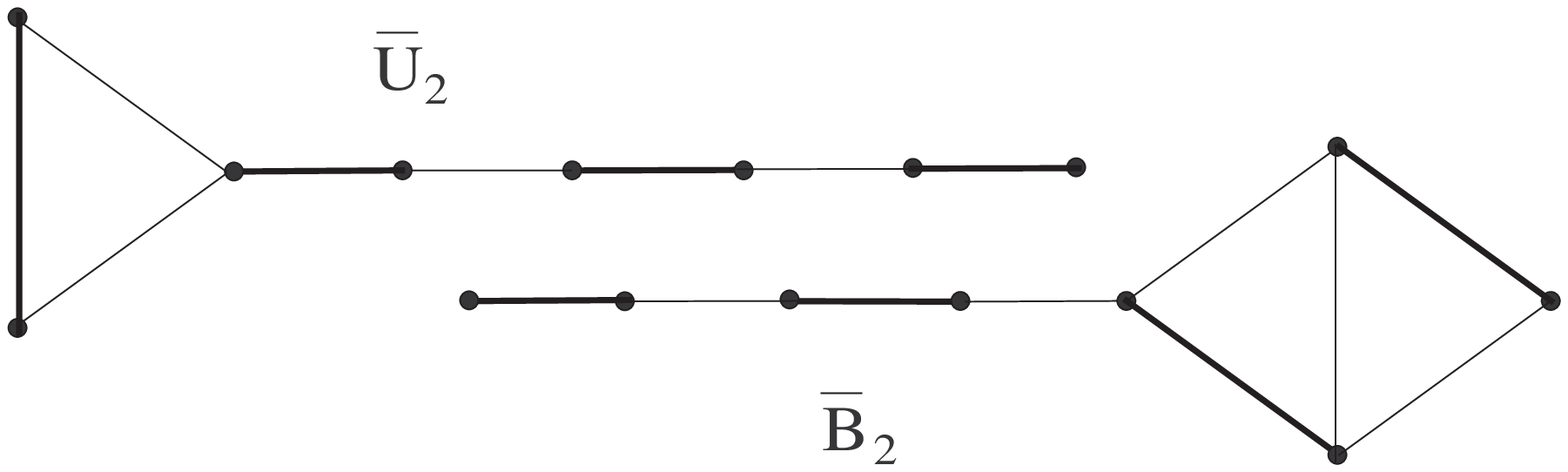}
\caption{The $8$-vertex conjugated unicyclic and bicyclic graphs $\overline{U}_1$, $\overline{U}_2$, $\overline{B}_1$ and $\overline{B}_2$.}\label{ext conj}
\end{figure}

Now we give some previously known results on the center of a graph from \cite{Harary53} and some results on extremal graphs with respect to total-eccentricity index from \cite{Farooq2017}. 
In next theorem, we give a result dealing with the location of center in a connected graph.
\begin{theorem}[Harary and Norman \cite{Harary53}]\label{Harary}
The center of a connected graph $G$ is contained in a block of $G$.
\end{theorem}
The only possible blocks in a unicyclic graph are $K_1$, $K_2$ or a cycle $C_k$. Thus the following corollary gives the center of an $n$-vertex conjugated unicyclic graph $\overline{U}$.
\begin{corollary}\label{Harary2}
If $\overline{U}$ is an $n$-vertex conjugated unicyclic graph with a unique cycle $C_k$, then $C(\overline{U}) = K_1$ or $K_2$, or $C(\overline{U}) \subseteq C_k$
\end{corollary}

The following results give the extremal unicyclic and bicyclic graphs with respect to total-eccentricity index.

\begin{theorem}[Farooq et al.~\cite{Farooq2017}]
\label{tau max uni}
Among all $n$-vertex unicyclic graphs, $n\geq 4$, the graph $U_2$ shown in Figure~\ref{max conj} has maximal total-eccentricity index.
\end{theorem}
\begin{theorem}[Farooq et al.\cite{Farooq2017}]
\label{tau max bi}
Among all $n$-vertex bicyclic graphs, $n\geq 5$, the graph $B_2$ shown in Figure~\ref{max conj} has the maximal total-eccentricity index.
\end{theorem}
In Section 2, we find extremal conjugated unicyclic and bicyclic graphs with respect to total-eccentricity index.

\section{Conjugated unicyclic and bicyclic graphs}

In this section, we find extremal conjugated unicyclic and bicyclic graphs with respect to total-eccentricity index. 
%
%
%
%
In~\eqref{U_B_1}, we give the total-eccentricity index of the conjugated graphs 
$\overline{U}_1$, $\overline{U}_2$, $\overline{B}_1$ and $\overline{B}_2$ which can easily be computed.
\begin{equation}\label{U_B_1}
\begin{array}{cc}
\tau(\overline{U}_1) = \frac{7}{2}n - 3, 
&
\tau(\overline{U}_2) = \frac{3n^2}{4} - n - \frac{3}{4}, \\[.2cm]
\tau(\overline{B}_1) = \frac{7}{2}n - 4,
&
\tau(\overline{B}_2) = \frac{3}{4}n^2 - n - 2.
\end{array}
\end{equation}
Using Theorem~\ref{Harary} and Corollary~\ref{Harary2}, we prove the following result.
\begin{remark}\label{Theorem 3.1}
When $n=4$, the graph shown in Figure~\ref{uni_n_46}(a) has the smallest total-eccentricity index among all $4$-vertex conjugated unicyclic graphs. 
When $n=6$, the graphs shown in Figure~\ref{uni_n_46}(b) and Figure~\ref{uni_n_46}(c) have smallest total-eccentricity index among $6$-vertex conjugated unicyclic graphs.
When $n=8$, the graph shown in Figure~\ref{uni_n_46}(d) has smallest total-eccentricity index among $8$-vertex conjugated unicyclic graphs.
\end{remark}
\begin{figure}[h!]
\centering
\includegraphics[height=3cm,width=12cm]{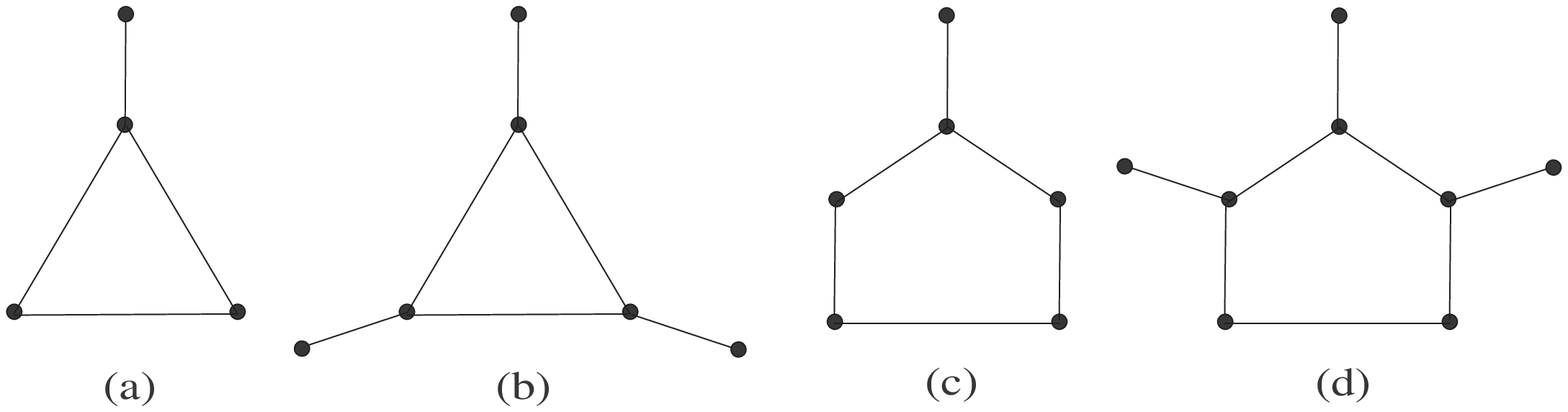}
\caption{The $n$-vertex conjugated unicyclic graphs with minimal total-eccentricity index when $n=4,6,8$.}\label{uni_n_46}
\end{figure}
\begin{theorem}\label{tau conj uni min}
Let $n\equiv 0\pmod 2$ and $n\geq 10$. Then among all $n$-vertex conjugated unicyclic graphs, the graph $\overline{U}_1$ shown in Figure~\ref{ext conj} has the minimal total-eccentricity index.
\end{theorem}
\begin{proof}
Let $\overline{U}_1$ be the $n$-vertex conjugated unicyclic graph shown in Figure~\ref{ext conj}.
Let $\overline{U}$ be an arbitrary $n$-vertex conjugated unicyclic graph with a unique cycle $C_k$.
We show that $\tau(\overline{U}) \geq \tau(\overline{U_1})$.
Let $n_i$ denote the number of vertices with eccentricity $i$ in $\overline{U}$.
If $k \geq 8$ or ${\rm rad}(\overline{U}) \geq 4$ then
\begin{equation*}
\tau(\overline{U}) \geq 4n > \frac{7n}{2}-3 = \tau(\overline{U}_1).
\end{equation*}
In the rest of the proof, we assume that $k \in \{3,4,5,6,7\}$ and ${\rm rad}(\overline{U}) \in \{2,3\}$.
Let $k\in \{6,7\}$. If $x$ is a vertex of $\overline{U}$ such that $x$ is not on $C_k$, then $e_{\overline{U}}(x) \geq 4$.
Also, it is easily seen that there are at most 
five vertices on $C_k$ with eccentricity $3$.
Thus
\begin{equation}\label{ecc_k67}
\tau(\overline{U}) \geq 3(5) + 4(n-5) = 4n - 5 > \frac{7n}{2}-3 = \tau(\overline{U}_1).
\end{equation}
We complete the proof by considering the following cases.\\
\textbf{Case 1}.
When ${\rm rad}(\overline{U}) = 3$ and $k\in \{3,4,5\}$.
By Corollary~\ref{Harary2}, $C(\overline{U}) = K_1$ or $C(\overline{U}) = K_2$ or $C(\overline{U}) \subseteq C_k$.
This shows that $\overline{U}$ has at most five vertices with eccentricity $3$.
Thus the inequality~\eqref{ecc_k67} holds in this case.\\
\textbf{Case 2}.
When ${\rm rad}(\overline{U}) = 2$ and $k\in \{4,5\}$. Then ${\rm diam}(\overline{U}) \leq 2 {\rm rad}(\overline{U}) = 4$ and there will be exactly one vertex in $C(\overline{U})$. That is, $n_2 = 1$. Let $v$ be the vertex with $e_{\overline{U}}(v) = 2$. Then $v\in V(C_k)$. Considering several possibilities for longest possible paths (of length $2$) starting from $v$ and that $\overline{U}$ is conjugated, one can see that $\overline{U}$ is isomorphic to one of the graphs shown in Figure~\ref{r=2_k=5}. Moreover, observe that $n_4\geq \frac{n}{2}-1$.
\begin{figure}[h]
\centering
\includegraphics[height=3.5cm,width=8.4cm]{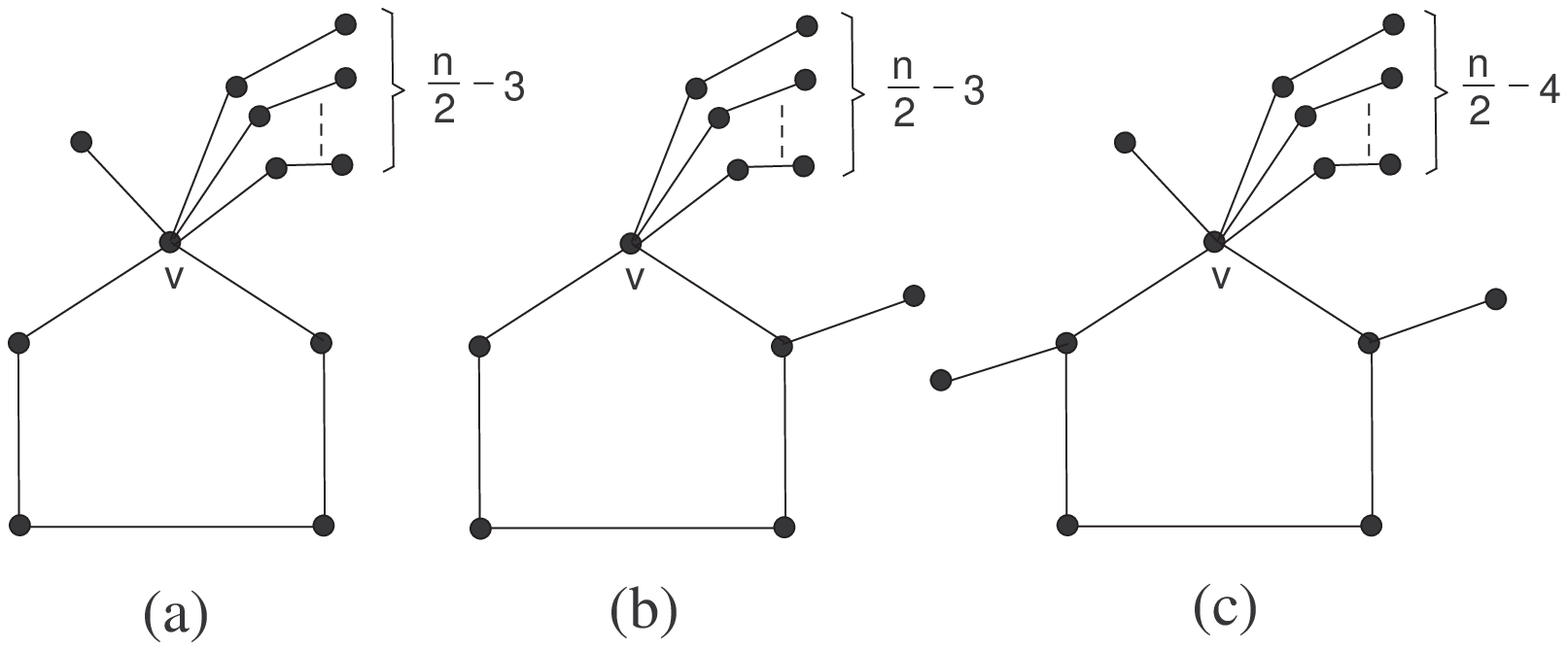}
\includegraphics[height=3.5cm,width=7cm]{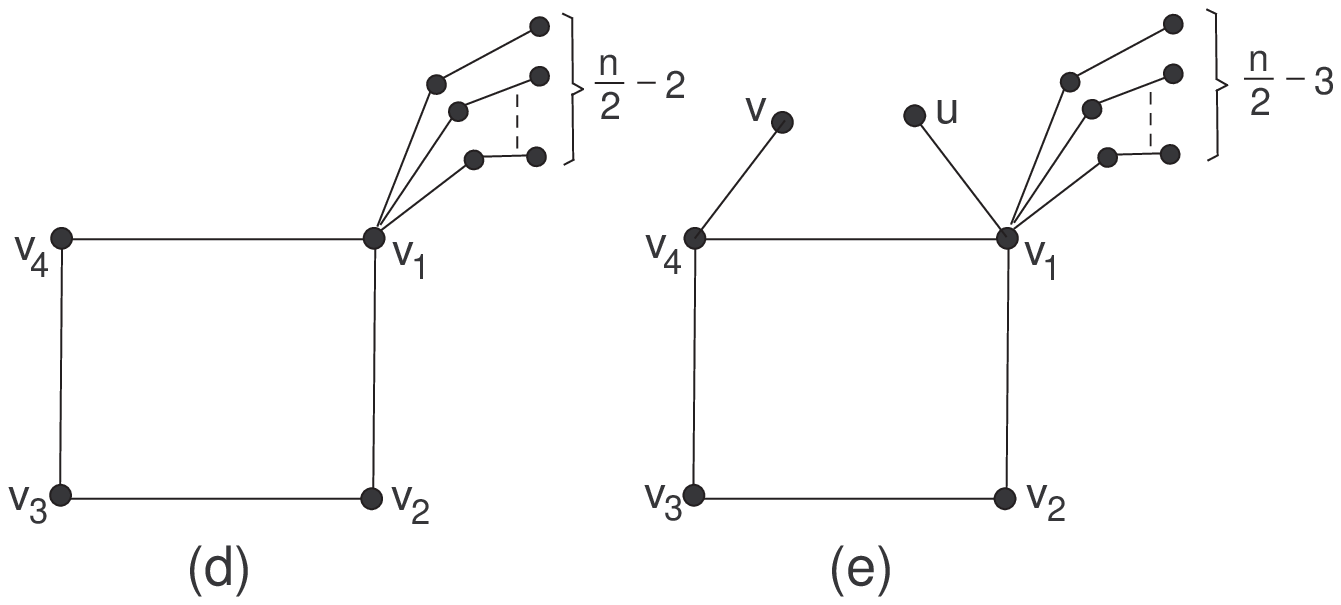}
\caption{The $n$-vertex conjugated unicyclic graphs discussed in Case 2.}\label{r=2_k=5}
\end{figure}

\noindent Since $n_2 + n_3 + n_4 = n$ and $n_2 = 1$, we can write
\begin{eqnarray*}
\tau(\overline{U}) &=& 2n_2 + 3n_3 + 4n_4\\
&=& 2 + 3(n_3 + n_4) + n_4\\
&=& 3n - 1 + n_4\\
&\geq& 3n - 1 + \frac{n}{2} - 1
= \frac{7n}{2} - 2 >  \tau(\overline{U}_1).
\end{eqnarray*}
%
%
\\
\textbf{Case 3.}
When ${\rm rad}(\overline{U}) = 2$ and $k=3$.
Then $n_2 = 1$ (see Figure~\ref{r=2_k=3}).
Let $v$ be the unique central vertex of $\overline{U}$.
Then either $v$ is a vertex of $C_3$ or $v$ is adjacent to a vertex of $C_3$.
%
%
%
When $v\in V(C_3)$, then $\overline{U}$ is isomorphic to one of the graphs shown in Figure~\ref{r=2_k=3}(a), \ref{r=2_k=3}(b) or \ref{r=2_k=3}(c). In this case, all vertices with eccentricity $4$ are pendent.
This gives
$n_4\geq \frac{n}{2} - 2$.
Therefore
\begin{eqnarray*}
\tau(\overline{U}) &=& 2 (1) + 3(n_3+n_4) + n_4\\
&\geq& 3n - 1 + \frac{n}{2} - 2\\
&=&\frac{7n}{2} - 3 = \tau(\overline{U}_1).
\end{eqnarray*}
Similarly, if the central vertex $v$ is not on $C_3$, then $\overline{U}$ is isomorphic to one of the graphs shown in Figure~\ref{r=2_k=3}(d) or \ref{r=2_k=3}(e).
Note that $n_4 \geq \frac{n}{2}$. Thus
$\nonumber \tau(\overline{U}) \geq 2 (1) + 3(n - 1) + \frac{n}{2}
= \frac{7n}{2}-1 > \tau(\overline{U}_1).$

Combining all the cases, we see that $\overline{U}_1$ is the minimal graph with respect to total-eccentricity index. This completes the proof.
\begin{figure}[h]
\centering
\includegraphics[width=12cm]{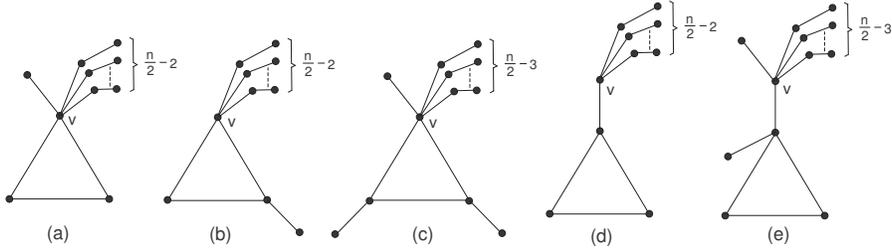}
\caption{The $n$-vertex conjugated unicyclic graphs discussed in Case 3.}\label{r=2_k=3}
\end{figure}
\end{proof}
The following theorem gives the maximal conjugated unicyclic graphs with respect to total-eccentricity index.
\begin{theorem}\label{tau conj uni max}
Let $n\equiv 0(\bmod 2)$. Then the $n$-vertex conjugated unicyclic graph corresponding to the maximal total-eccentricity index is the graph $\overline{U}_2$ shown in Figure~\ref{ext conj}.
\end{theorem}
\begin{proof}
Note that the class of all $n$-vertex conjugated unicyclic graphs forms a subclass of the class of all $n$-vertex unicyclic graphs. From Theorem~\ref{tau max uni} we see that among all $n$-vertex unicyclic graphs, the graph $U_2$ (see Figure~\ref{max conj}) has the largest total-eccentricity index.
Since $U_2$ admits a a perfect matching when $n\equiv 0(\bmod 2)$, the result follows.
%
\end{proof}
%
%
%
%
\begin{corollary}
For an $n$-vertex conjugated unicyclic graph $\overline{U}$, we have $\frac{7n}{2} - 3 \leq \tau(\overline{U}) \leq \frac{3n^2}{4} - n - \frac{3}{4}$.
\end{corollary}
\begin{proof}
Using Theorem~\ref{tau conj uni min}, Theorem~\ref{tau conj uni max} and equation~\eqref{U_B_1}, we obtain the required result.
\end{proof}

The next theorem gives the minimal conjugated bicyclic graphs with respect to total-eccentricity index.
\begin{remark}\label{Theorem 3.4}
Let $n = 4$. Then among all $4$-vertex conjugated bicyclic graphs,
one can easily see that the graph shown in Figure~\ref{counterbi}(a) has the minimal total-eccentricity index.
Similarly, when $n=6$ and $8$, then the graphs respectively shown in Figure~\ref{counterbi}(b) and Figure~\ref{counterbi}(c) have the minimal total-eccentricity index among all $6$-vertex and $8$-vertex conjugated bicyclic graphs.
\end{remark}
\begin{figure}[h!]
\centering
\includegraphics[width=7cm]{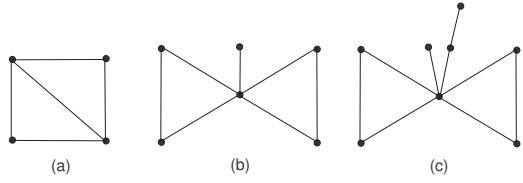}
\caption{The $n$-vertex conjugated bicyclic graphs with minimal total-eccentricity index when $n=4,6$ and $8$.}\label{counterbi}
\end{figure}
\begin{theorem}\label{tau conj bi min}
Let $n\equiv 0\pmod 2$ and $n\geq 10$. Then among the $n$-vertex conjugated bicyclic graphs, the graph $\overline{B}_1$ shown in Figure~\ref{ext conj} has the minimal total-eccentricity index.
\end{theorem}
\begin{proof}
%
Let $\overline{B}_1$ be the $n$-vertex conjugated bicyclic graph shown in Figure~\ref{ext conj}.
%
%
Let $\overline{B}$ be an arbitrary $n$-vertex conjugated bicyclic graph and $\overline{B} \ncong \overline{B}_1$.  Let $C(\overline{B})$ denote the center of $\overline{B}$ and $n_i$ denote the number of vertices with eccentricity $i$.
The proof is divided into two cases depending upon the number of cycles in $\overline{B}$.
\\
\textbf{Case 1}.
When $\overline{B}$ contains two edge-disjoint cycles $C_{k_1}$ and $C_{k_2}$ of lengths $k_1$ and $k_2$, respectively.
Without loss of generality, assume that $k_1\leq k_2$.
%
%
%
If ${\rm rad}(\overline{B})\geq 4$ or $k_2\geq 8$, then
$$\tau(\overline{B}) \geq 4n > \frac{7n}{2}-4 = \tau(\overline{B}_1).$$
Thus, we assume that $k_2\in \{3,4,5,6,7\}$ and ${\rm rad}(\overline{B}) \in \{2,3\}$.
If $k_2\in \{6,7\}$, then for any vertex $x \notin V(C_{k_2})$, $e_{\overline{B}}(x) \geq 4$. 
Moreover, as $k_2\leq 7$, the number of vertices with eccentricity $3$ are at most $7$.
Thus
\begin{equation}\label{T67}
\tau(\overline{B}) \geq 3(7) + 4(n-7) = 4n - 7 > \frac{7n}{2}-4 = \tau(\overline{B}_1).
\end{equation}
%
%
%
We consider the following three subcases.
\\[.2cm]
\textbf{(a)}
Let ${\rm rad}(\overline{B})=3$ and $k_2 \in \{3,4,5\}$. By Theorem~\ref{Harary}, we have $|V(C(\overline{B}))| \leq 7$. Thus $\tau(\overline{B})$ satisfies equation~\eqref{T67}.
\\[.2cm]
\textbf{(b)}
Let ${\rm rad}(\overline{B})=2$ and $k_2\in \{4,5\}$.
Take $v\in V(C(\overline{B}))$.
We observe that the center $C(\overline{B})$ is contained in $C_{k_2}$ and $n_2  = 1$. 
Assume $v$ to be the unique central vertex of $\overline{B}$.
Then for several possible choices for possible pendent vertices in the conjugated graph $\overline{B}$, one can observe that $\overline{B}$ is one of the graphs shown in Figure~\ref{r=2_k=45_bi}.
Moreover $n_4 \geq \frac{n}{2}-2$.
Then $\tau(\overline{B}) \geq 2n_2 + 3n_3 + 4n_4 = 2n_2 + 3(n_3 + n_4) + n_4 \geq 2(1) + 3(n-1) + \frac{n}{2}-2 = \frac{7n}{2} - 3 > \tau(\overline{B}_1)$.
\begin{figure}[h!]
\centering
\includegraphics[height=4.9cm,width=11cm]{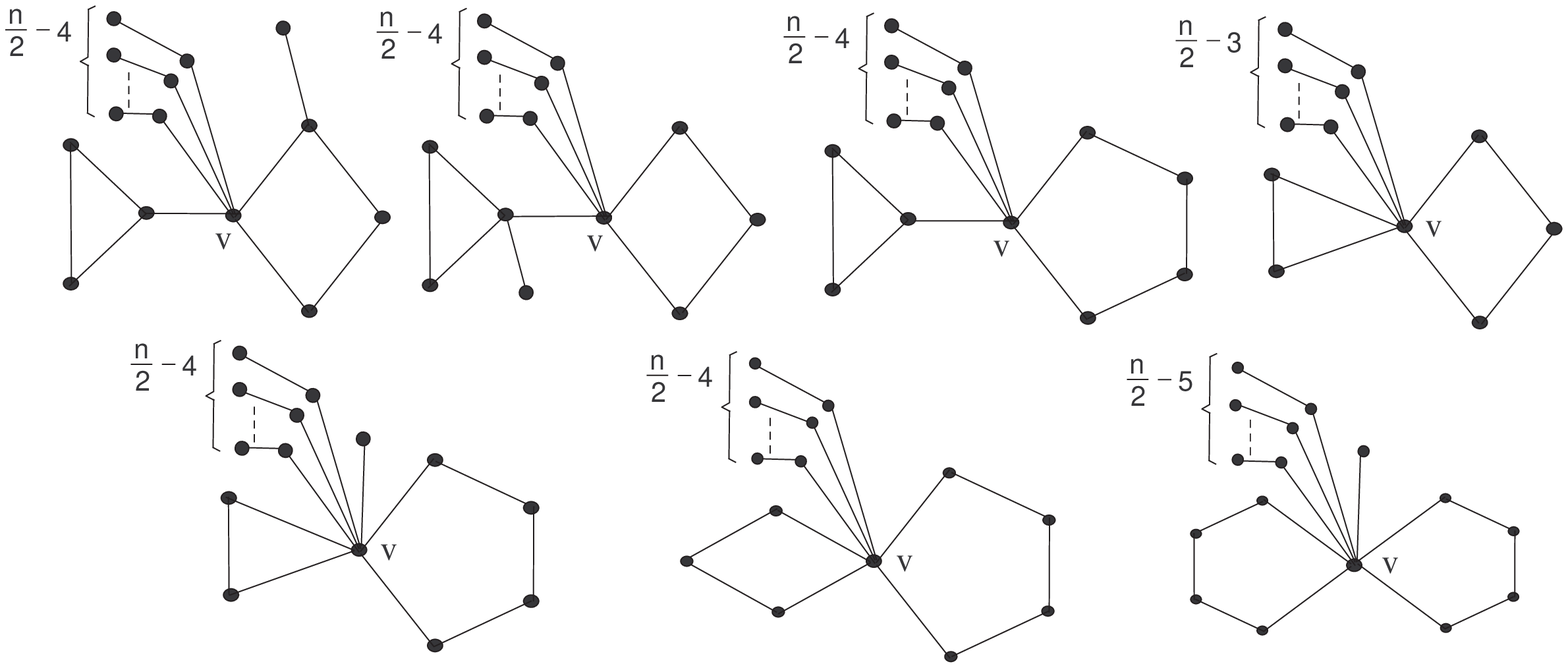}
\caption{The $n$-vertex conjugated bicyclic graphs discussed in Case 1-(b).}\label{r=2_k=45_bi}
\end{figure}
\\[.2cm]
\textbf{(c)}
Let ${\rm rad}(\overline{B})=2$ and $k_1 = k_2 = 3$. 
If $C(\overline{B}) \subseteq C_{k_2}$ then $C(\overline{B}) = K_1$, otherwise $n\leq 8$ which is not true.
Similarly, if $C(\overline{B}) \nsubseteq C_{k_2}$, then $C(\overline{B}) = K_2$ or $K_1$.
Since $n_2 \ngeq 2$, we have $n_2 = 1$.
Let $c$ be the unique central vertex.
Then $\overline{B}$ is isomorphic to one of the graphs shown in Figure~\ref{counterbi-2}.
When $c$ is not a vertex of $C_{k_1}$ or $C_{k_2}$,
then $n_4 \geq \frac{n}{2}+1$ and $\tau(\overline{B}) \geq 2(1) + 3(n-1) + \frac{n}{2}+1 = \frac{7n}{2} > \tau(\overline{B}_1)$.
If $c$ is a vertex of $C_{k_1}$ or $C_{k_2}$, then $n_4\geq \frac{n}{2}-3$.
Thus
$\tau(\overline{B}) \geq 2(1) + 3(n-1) + \frac{n}{2}-3 = \frac{7n}{2} - 4 = \tau(\overline{B}_1)$.
\begin{figure}[h]
\centering
\includegraphics[height=3cm,width=8.5cm]{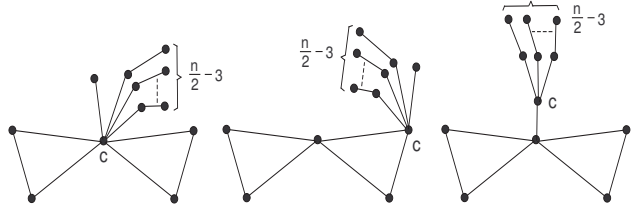}
\caption{The $n$-vertex conjugated bicyclic graphs studied in Case 1-(c).}\label{counterbi-2}
\end{figure}
\\
\textbf{Case 2.}
When cycles of $\overline{B}$ share some edges. Then there are cycles $C_{k_1}$, $C_{k_2}$ and $C_{k_3}$ in $\overline{B}$ of lengths $k_1$, $k_2$ and $k_3$, respectively.
Without loss of generality, assume that $k_1\leq k_2\leq k_3$.
Then $k_1,k_2 \geq 3$ and $k_3 \geq 4$.
%
%
%
%
Let $Q$ be the subgraph of $\overline{B}$ induced by the vertices of $C_{k_1}$, $C_{k_2}$ and $C_{k_3}$.
%
Clearly, $Q$ contains the cycle $C_{k_3}$.
%
Assume that $V(C_{k_3}) = \{v_1,v_2,\ldots,v_{k_3}\}$.
Then $\overline{B}$ is minimal with respect to total-eccentricity index if $Q$ can be obtained from the cycle $C_{k_3}$ by adding the edge $v_1v_{\lfloor\frac{k_3}{2} \rfloor + 1}$ or $v_1v_{\lfloor\frac{k_3}{2} \rfloor + 2}$.
%

When $k_3 \geq 13$, then $e_{\overline{B}}(w) \geq 4$ for all $w\in V(\overline{B})$.
This gives us $\tau(\overline{B}) \geq 4n > \tau(\overline{B}_1)$.
%
Assume that ${\rm rad}(\overline{B}) \leq 3$ and $k_3\leq 12$.
We consider the following two subcases:\\[.2cm]
\textbf{(a)} When $k_3\in \{9,10,11,12\}$. 
Then ${\rm rad}(\overline{B}) = 3$.
There exists a vertex $c$ with degree $3$ in $Q$ such that $c\in C(\overline{B})$ and $c$ has at most $\frac{n-10}{2}+1$ neighbours not in $Q$. Clearly, all of these vertices will have eccentricity $4$.
Then $\frac{n}{2}-5$ such vertices can have unique pendent vertices with eccentricity $5$. Moreover, there will be at least one vertex with eccentricity $5$ in $Q$, otherwise ${\rm rad}(\overline{B}) \neq 3$. Thus $n_5 \geq \frac{n}{2}-4$.
Moreover, at most $3$ vertices in $Q$ can have eccentricity $3$. This gives $1 \leq n_3 \leq 3$.
When $k_3=9$, such a graph $\overline{B}$ is shown in Figure~\ref{r=2_k3=4}(a). The vertex $v\in V(Q)$ such that $e_{\overline{B}}(v) = 5$ is also shown in the figure.
Using the facts that $n = n_3+n_4+n_5$  and $n_4+n_5 = n - n_3 \geq n - 3$.
We get 
\begin{eqnarray*}
\tau(\overline{B}) &=& 3n_3 + 4n_4 + 5n_5 \\
&=& 3n_3 + 4(n_4 + n_5) + n_5 \\
&\geq& 3(1) + 4(n - 3) + \frac{n}{2} - 4 \\
&=& \frac{9n}{2} - 13 \geq \tau(\overline{B}_1).
\end{eqnarray*}
%
\textbf{(b)}
When $k_3 \in \{4,5,6,7,8\}$. 
We first assume that ${\rm rad}(\overline{B}) = 3$.
Then by Theorem~\ref{Harary}, we have $1 \leq n_3 \leq 8$. Thus
$\tau(\overline{B}) \geq 3(8) + 4(n-8) = 4n - 8 \geq \tau(\overline{B}_1)$.
On the other hand, when ${\rm rad}(\overline{B}) = 2$. Then $n_2 = 1$.
When $C(\overline{B}) \nsubseteq Q$, then the minimal graph $\overline{B}$ is isomorphic to the graph shown in Figure~\ref{r=2_k3=4}(b). Clearly, $n_4 \geq \frac{n}{2}$. Thus $\tau(\overline{B}) = 2n_2 + 3(n_3 + n_4) + n_4 \geq 2(1) + 3(n-1) + \frac{n}{2} = \frac{7n}{2} - 1 > \tau(\overline{B}_1)$.
When $C(\overline{B}) \subseteq Q$, then $\overline{B}$ is isomorphic to one of the graphs shown in Figures~\ref{r=2_k3=4}(c)$-$\ref{r=2_k3=4}(h). 
It can be seen that $n_4\geq \frac{n}{2}-2$. Thus we can write $\tau(\overline{B}) = 2n_2 + 3(n_3 + n_4) + n_4 \geq 2(1) + 3(n-1) + \frac{n}{2} - 2 = \frac{7n}{2} - 3 = \tau(\overline{B}_1)$.
\begin{figure}[h!]
\centering
\includegraphics[width=15.7cm]{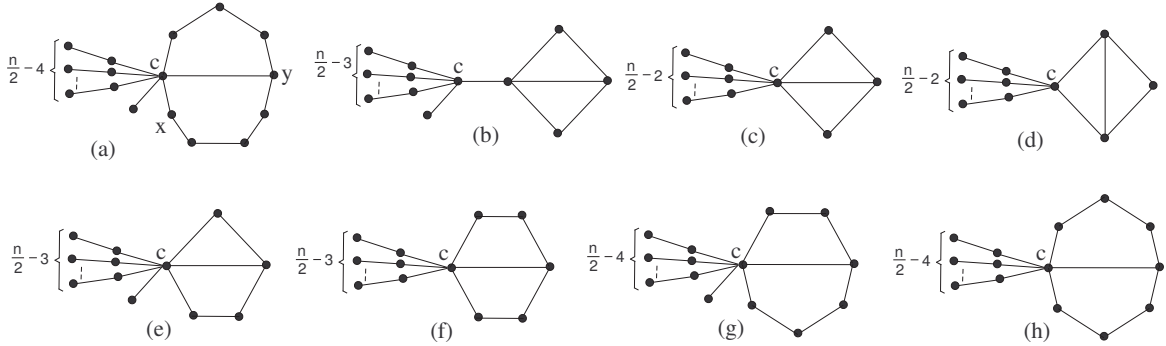}
\caption{The conjugated bicyclic graphs studied in Case 2. 
The vertices $c,x$ and $y$ represent central vertices.
}\label{r=2_k3=4}
\end{figure}
\\
Combining the results of Case 1 and Case 2, we see that among all conjugated bicyclic graphs, $\overline{B}_1$ has the minimal total-eccentricity index.
The proof is complete.
\end{proof}
\begin{theorem}\label{tau conj bi max}
Let $n\equiv 0\pmod 2$. Then among the $n$-vertex conjugated bicyclic graphs, the graph $\overline{B}_2$ shown in Figure~\ref{ext conj} has the maximal total-eccentricity index.
\end{theorem}
\begin{proof}
For $n\equiv 0(\bmod 2)$, the proof can be derived from the proof of Theorem~\ref{tau max bi}.
\end{proof}
\begin{corollary}
For any conjugated bicyclic graph $\overline{B}$, we have $\frac{7n}{2}-4 \leq \tau(\overline{B}) \leq \frac{3n^2}{4} - n -2$.
\end{corollary}
\begin{proof}
The result follows by using Theorem~\ref{tau conj bi min}, Theorem~\ref{tau conj bi max} and equation~\eqref{T67}.
\end{proof}

\section{Conclusion}
In this paper, we extended the results of Farooq et al.~\cite{Farooq2017} and studied the extremal conjugated unicyclic and bicyclic graphs with respect to total-eccentricity index.

\section*{Acknowledgements}
The authors are thankful to the Higher Education Commission of Pakistan for supporting this research under the grant 20-3067/NRPU/R$\&$D/HEC/12/831.

\end{document}